\documentclass[journal,twoside,web]{ieeecolor}
\definecolor{subsectioncolor}{RGB}{0,0,0}
\usepackage{generic}
\usepackage{multirow}
\usepackage{booktabs}
\usepackage{mathrsfs}
\usepackage{cite}
\usepackage{amsmath,amssymb,amsfonts}
\usepackage[mathscr]{euscript}
\usepackage{graphicx}
\usepackage[hidelinks]{hyperref}
\usepackage[capitalize,noabbrev]{cleveref}
\usepackage{textcomp}
\def\BibTeX{{\rm B\kern-.05em{\sc i\kern-.025em b}\kern-.08em
    T\kern-.1667em\lower.7ex\hbox{E}\kern-.125emX}}
\usepackage{url}
\usepackage{pifont}
\usepackage{cleveref}
\usepackage{algorithm}
\usepackage{algpseudocode}

\crefname{equation}{Eq.}{Eqs.}
\Crefname{equation}{Equation}{Equations}
\crefname{theorem}{Theorem}{Theorems}
\Crefname{lemma}{Lemma}{Lemmas}

\newtheorem{proposition}{Proposition}

\newtheorem{definition}{Definition}

\newcommand{\rank}{\text{rank}}

\newcommand{\Bv}{\mathbf{B}}

\newcommand{\Nv}{\mathbf{N}}

\newcommand{\Rv}{\mathbf{R}}

\newcommand{\At}{\mathscr{A}}

\newcommand{\As}{\mathcal{A}}

\newcommand{\Ds}{\mathcal{D}}
\newcommand{\Es}{\mathcal{E}}

\newcommand{\Gs}{\mathcal{G}}
\newcommand{\Hs}{\mathcal{H}}

\newcommand{\Os}{\mathcal{O}}

\newcommand{\Qs}{\mathcal{Q}}

\newcommand{\Ss}{\mathcal{S}}
\newcommand{\Ts}{\mathcal{T}}

\newcommand{\Vs}{\mathcal{V}}
\newcommand{\Ws}{\mathcal{W}}

\begin{document}
\title{\fontsize{22}{20}\selectfont Structural Controllability of Large-Scale Hypergraphs}
\author{Joshua Pickard, Xin Mao, and Can Chen
\thanks{Joshua Pickard is supported by the Eric and Wendy Schmidt Center at the Broad Institute of MIT and Harvard  (Corresponding authors: Joshua
Pickard; Can Chen).}
\thanks{Joshua Pickard is with the Eric and Wendy Schmidt Center, Broad Institute of MIT and Harvard, Cambridge, MA 02142, USA (e-mail: jpickard@broadinstitute.org).}
\thanks{Xin Mao is with the School of Data Science and Society, University of North Carolina at Chapel Hill, Chapel Hill, NC 27599, USA (e-mail: xinm@unc.edu).}
\thanks{Can Chen is with the School of Data Science and Society, the Department of Mathematics, and the Department of Biostatistics, University of North Carolina at Chapel Hill, Chapel Hill, NC 27599, USA (e-mail: canc@unc.edu).}}

\maketitle
\thispagestyle{empty}
\pagestyle{empty}

\begin{abstract}
    Controlling real-world networked systems, including ecological, biomedical, and engineered networks that exhibit higher-order interactions, remains challenging due to inherent nonlinearities and large system scales. Despite extensive studies on graph controllability, the controllability properties of hypergraphs remain largely underdeveloped. Existing results  focus primarily on exact controllability, which is often impractical for large-scale hypergraphs. In this article, we develop a structural controllability framework for hypergraphs by modeling hypergraph dynamics as polynomial dynamical systems. In particular, we extend classical notions of accessibility and dilation from linear graph-based systems to polynomial hypergraph dynamics and establish a hypergraph-based criterion under which the topology guarantees satisfaction of classical Lie-algebraic and Kalman-type rank conditions for almost all parameter choices. We further derive a topology-based lower bound on the minimum number of driver nodes required for structural controllability and leverage this bound to design a scalable driver node selection algorithm combining dilation-aware initialization via maximum matching with greedy accessibility expansion. We demonstrate the effectiveness and scalability of the proposed framework through numerical experiments on hypergraphs with tens to thousands of nodes and higher-order interactions.
\end{abstract}


\section{Introduction}
Control of large-scale networked systems arises in a wide range of application domains, including biomedical, social, power, and ecological systems \cite{li2019control, kandhway2016using,ouammi2014optimal,liu2024control}.
In these networks, interventions applied to a small subset of nodes can influence the global system behavior, making network control a critical problem.
Initiated by C. T. Lin in 1974, substantial progress has been made in graph controllability, giving rise to structural controllability and driver node selection methods that enable scalable analysis and control design for networked systems \cite{lin1974structural, shields1976structural, mayeda1979strong, liu2011controllability, mousavi2017structural, menara2018structural,d2023controlling}.
However, many real-world networked systems exhibit higher-order interactions, in which the evolution of each node depends simultaneously on multiple nodes \cite{benson2016higher, zhu2018social, higham2021epidemics, dotson2022deciphering, bick2023higher, alge2024continuous,delabays2025hypergraph}.
For instance, in ecological networks, species often engage in higher-order interactions, where the relationship between two species is influenced by the presence of one or more additional species \cite{golubski2016ecological}.
In such settings, graph-based representations and traditional graph controllability methods are insufficient to capture and control the underlying interaction structure.

Hypergraphs generalize graphs by allowing hyperedges to connect more than two nodes, providing natural representations of group interactions \cite{berge1984hypergraphs}.
The adjacency structure of a hypergraph admits a representation as a higher-order tensor (i.e., a multidimensional array).
This representation enables the modeling of dynamics on hypergraphs via tensor-based dynamical systems, which are equivalent to homogeneous polynomial  systems.
Recent work has extended classical control theory notions of controllability, observability, and other systems properties and methods to hypergraph-induced polynomial dynamics \cite{chen2021controllability,pickard2023observability, pickard2023kronecker, pickard2024geometric, pickard2025scalable,dong2024controllability,chen2022explicit,10910206,11124400,cui2024discrete,zhang2024global,mao2026model,11240158}.
In particular, Chen et al. developed a tensor-based extension of Kalman’s rank condition to determine the exact controllability of hypergraphs \cite{chen2021controllability}.
This condition verifies the classical Lie algebraic conditions for nonlinear controllability in the hypergraph setting \cite{hermann1977nonlinear, jurdjevic1983polynomial, jurdjevic1985polynomial}, and was combined with a greedy heuristic to identify minimal driver node sets and to elucidate how higher-order interactions affect controllability.
However, verifying the generalized Kalman-type condition is computationally intensive, with time complexity that grows rapidly with both  system dimension and  interaction order.
Moreover, exact controllability analysis depends on absolute knowledge of all system parameters \cite{lee1967foundations}.
These constraints  hinder the applicability of existing hypergraph controllability methods to large-scale real-world systems, where interaction strengths are often uncertain, heterogeneous, or inferred from data.

Structural controllability offers an alternative perspective to exact controllability, addressing issues of parameter uncertainty and algorithmic scalability, by considering only the interaction topology of the system independent of the parameter values.
This perspective is appealing for large-scale networked systems, where interaction strengths are often uncertain \cite{liu2011controllability, chapman2013strong, liu2016control, lee2020heritability}.
The concept characterizes linear controllability based on the zero-nonzero pattern of the system matrices, bypassing the need for algebraic tests such as the Kalman's rank condition or the Popov--Belevitch--Hautus (PBH) test \cite{gilbert1963controllability, kalman1963mathematical, kalman1968lectures, popov1973hyperstability}.
By encoding the system structure as a directed graph, this relaxation enables the use of graph-theoretic tools to analyze controllability and to design actuator placement strategies efficiently.
Liu et al. demonstrated that these methods scale to networks with hundreds of thousands of nodes and can guide the selection of a minimal set of driver nodes through maximum matching  to achieve controllability \cite{liu2011controllability}.
As a result, structural controllability has become a cornerstone of graph control theory.
Despite these successes, an analogous structural theory for hypergraph dynamics remains largely undeveloped.
Developing such a theory is challenging due to the  nonlinear nature of higher-order interactions and combinatorial complexity of hypergraph structures.

In this article, we develop a novel structural controllability framework for large-scale hypergraphs.
We first highlight the limitations imposed by system identification and numerical representation of applying the Kalman-type condition for hypergraph controllability.
Building on Chen’s work on  hypergraph controllability \cite{chen2021controllability}, we model hypergraph dynamics using tensor-based homogeneous polynomial  systems and demonstrate that structural controllability can be characterized through accessibility and dilation in hypergraphs, which generalize the corresponding concepts from linear systems and graph theory.
Unlike exact controllability approaches, this framework does not rely on precise knowledge of system parameters, making it robust to uncertainty and heterogeneity in interaction strengths.
We further derive a topology-based lower bound on the minimum number of driver nodes required to achieve structural controllability and propose scalable driver node selection algorithms that combine dilation-aware initialization via maximum matching with greedy accessibility expansion.
Numerical experiments demonstrate the effectiveness and scalability of our framework  across a wide range of  hypergraphs with tens of thousands of nodes and interactions.

Beyond its theoretical contributions, our framework has practical implications for domains such as ecological network management, biological regulation, and engineered infrastructure systems, where higher-order interactions are prevalent, control resources are limited, and system parameters are often uncertain.
In ecological networks, for instance, higher-order interactions among species, such as facilitation, competition, or mutualistic effects mediated by shared resources, are difficult to quantify accurately, making parameter-dependent controllability analysis infeasible \cite{bairey2016high,mayfield2017higher}.
In this context, structural driver node selection provides a principled approach to identify a minimal set of species or intervention points whose regulation can influence global system dynamics without requiring precise knowledge of interaction strengths.
More broadly, the proposed framework enables scalable controllability analysis and actuator placement for high-dimensional nonlinear polynomial dynamical systems, where exact controllability methods based on algebraic or differential-geometric criteria are computationally prohibitive.

The remainder of the article is organized as follows.
Section~\ref{sec:2} introduces preliminaries on tensors and hypergraphs, controllability of polynomial systems and hypergraphs, and structural controllability of linear graph-based systems.
Section~\ref{sec:3} presents the main theoretical results on hypergraph structural controllability, including system identification limitations, key hypergraph structures, the controllability framework, and scalable verification algorithms.
Section~\ref{sec:numerical_experiments} demonstrates the performance of the proposed driver node selection method through numerical experiments.
Section~\ref{sec:5.5} discusses extensions of the framework to broader classes of polynomial systems and control objectives.
Finally, Section~\ref{sec:5} provides concluding remarks with future directions.

\section{Preliminaries}\label{sec:2}

\subsection{Tensors and Hypergraphs}
A tensor is a higher-order extension of vectors and matrices \cite{kolda2009tensor,chen2024tensor}.
We consider a $k$th-order, $n$-dimensional tensor
$\mathscr{T}\in\mathbb{R}^{n\times n\times\stackrel{k}{\cdots}\times n}$ whose entries are indexed by  $\mathscr{T}_{j_1j_2\cdots j_k}$. The mode-$p$ tensor-vector multiplication of $\mathscr{T}$ with a vector $\textbf{v}\in\mathbb{R}^n$ is defined as
\begin{equation*}
(\mathscr{T} \times_{p} \textbf{v})_{j_1j_2\cdots j_{p-1}j_{p+1}\cdots j_k}=\sum_{j_p=1}^{n}\mathscr{T}_{j_1j_2\cdots j_p\cdots j_k}\textbf{v}_{j_p},
\end{equation*}
yielding a tensor of order $k-1$. 
Applying successive tensor-vector multiplication along the last $k-1$ modes results in
\begin{equation*}\label{eq: tensor homogeneous polynomial}
\begin{split}
\mathscr{T}\times_2 \textbf{v} \times_3\textbf{v} \times_4\cdots \times_{k}\textbf{v}\in\mathbb{R}^n,
\end{split}
\end{equation*}
which defines a homogeneous polynomial system of degree $k-1$. For notational simplicity, we denote it by $\mathscr{T}\textbf{v}^{k-1}$.
Moreover, the homogeneous polynomial can be equivalently expressed in matrix form as $\mathscr{T}\textbf{v}^{k-1}=\textbf{T}_{(1)}\textbf{v}^{[k-1]},$ where $\textbf{T}_{(1)}$ is the  matrix obtained by unfolding  $\mathscr{T}$ in the first mode, and
\begin{equation*}
\textbf{v}^{[k]}=\textbf v\otimes \textbf v\otimes\stackrel{k}{\cdots}\otimes \textbf v
\end{equation*}
denotes Kronecker exponentiation \cite{bellman1970introduction}.

Hypergraphs generalize graphs by allowing hyperedges to connect more than two nodes \cite{berge1984hypergraphs}.
An undirected hypergraph is defined as a pair $\mathcal{H}=\{\mathcal{V},\mathcal{E}\}$ where $\mathcal{V}=\{v_1,v_2,\dots,v_n\}$ is the set of nodes and  $\mathcal{E}$ is the set of hyperedges with each hyperedge $e\in\mathcal{E}$ satisfying $e\subseteq\mathcal{V}$ \cite{berge1984hypergraphs}.
Since linear systems correspond to first-order polynomial dynamics, extending graph-based models to hypergraphs naturally yields polynomial dynamical systems in which each hyperedge induces a polynomial term involving the states of its incident nodes.
In particular, for a $k$-uniform hypergraph in which each hyperedge contains exactly $k$ nodes, the resulting hypergraph dynamics is described by homogeneous polynomials of degree $k-1$.
These polynomial dynamics can be compactly represented using the adjacency tensor associated with the hypergraph.
Given an undirected hypergraph $\mathcal{H}$, the adjacency tensor $\mathscr{A}$ is defined as
\begin{equation*}\label{eq: adjacency tensor}
    \mathscr{A}_{j_1j_2\cdots j_k}=\begin{cases}
        c&\text{if } \{v_{j_1},v_{j_2},\dots,v_{j_k}\}\in\mathcal{E}\\
        0&\text{otherwise}
    \end{cases},
\end{equation*}
where the coefficient $c$ represents the interaction weight associated with the corresponding hyperedge.

\subsection{Controllability of Polynomial  Systems and Hypergraphs}
To study the controllability of hypergraph dynamics, we consider a homogeneous polynomial  system with linear control of the form
\begin{equation}\label{eq: polynomial with linear inputs}
    \dot{\textbf{x}}(t)=\textbf{f}\big(\textbf{x}(t)\big)+\textbf{B}\textbf{u}(t),
\end{equation}
where $\textbf{x}(t)\in\mathbb{R}^n$ is the system state, $\textbf{f}(\cdot)$ is an odd-degree homogeneous polynomial vector field (even-degree systems are generally uncontrollable), $\textbf{B}\in\mathbb{R}^{n\times m}$ is the control matrix, and $\textbf{u}(t)\in\mathbb{R}^m$ is the control signal.
The controllability of nonlinear polynomial systems of the form \eqref{eq: polynomial with linear inputs} has been extensively investigated from a geometric control perspective \cite{mayne1973geometric, hermann1977nonlinear, jurdjevic1981control, jurdjevic1983polynomial, baillieul1981controllability, jurdjevic1985polynomial, jurdjevic1997geometric}.
A central result is the Lie algebra rank condition, which provides a characterization of strong controllability for such systems.

\begin{proposition}[\cite{jurdjevic1997geometric}]\label{thm: jurdjevic and kupka}
   The homogeneous polynomial  system with linear input \eqref{eq: polynomial with linear inputs} is strongly controllable if and only if the Lie algebra spanned by the set of vector fields $\{\textbf{f},\textbf{b}_{1},\textbf{b}_{2},\dots,\textbf{b}_{m}\}$ is full rank, where $\textbf{b}_{j}$ are the $j$th columns of  $\textbf{B}$.
    Moreover, the Lie algebra has full rank $n$ at all points of $\mathbb{R}^n$ if and only if it is of full rank at the origin.
\end{proposition}

Since homogeneous polynomial systems can be equivalently represented using tensors, the system  \eqref{eq: polynomial with linear inputs} can be rewritten  as
\begin{equation}\label{eq: tensor with linear inputs}
    \dot{\textbf{x}}(t)=\mathscr{A}\textbf{x}(t)^{k-1}+\textbf{B}\textbf{u}(t),
\end{equation}
where $\mathscr A\in\mathbb{R}^{n\times n\times\stackrel{k}{\cdots}\times n}$ is a $k$th-order adjacency tensor with even $k$.
This formulation encodes higher-order interactions and provides a direct connection between polynomial dynamics and hypergraph representations.
Building on this representation, Chen et al. \cite{chen2021controllability} derived a generalized Kalman's rank condition to verify the exact controllability of tensor-based homogeneous polynomial systems with linear input.

\begin{proposition}[\cite{chen2021controllability}]\label{thm: cohg}
The tensor-based polynomial system (\ref{eq: tensor with linear inputs}) is strongly controllable if and only if the nonlinear controllability matrix  defined as
\begin{align*}
\textbf C=\begin{bmatrix}\textbf M_0 & \textbf M_1 &\cdots& \textbf M_{n-1}\end{bmatrix},
\end{align*}
where $\textbf M_0=\textbf B$, and 
each $\textbf M_j$  is formed from 
\begin{equation*}
\begin{split}
    \Big\{&\mathscr A\textbf v_1\textbf{v}_2\cdots\textbf v_{k-1}\text{ }\big|\text{ }\textbf v_{i}\in\mathrm{span}\left(\begin{bmatrix}\textbf M_0 & \textbf M_1 &\cdots& \textbf M_{j-1}\end{bmatrix}\right) \\ &\text{ for } i=1,2,\dots,k-1\Big\}
    \end{split}
\end{equation*}
for $j=1,2,\dots,n-1$, has full rank $n$.
\end{proposition}

The columns of $\textbf{C}$ span the Lie algebra evaluated at the origin, providing a Kalman-like test for strong controllability for polynomial dynamical systems.
Based on this condition, Chen et al. further studied the problem of selecting a minimum set of driver nodes for hypergraphs, whose control suffices to guarantee strong controllability \cite{chen2021controllability}.
Driver node selection is, in general, a combinatorial optimization problem, for which a greedy heuristic was proposed.
However, this approach faces both practical and theoretical limitations.
Practically, constructing  the nonlinear controllability matrix incurs substantial computational cost that grows rapidly with both system dimension and polynomial degree.
Each iteration of greedy driver node selection requires multiple tensor-vector multiplications or Kronecker exponentiations followed by singular value decompositions, rendering the approach computationally prohibitive for large-scale systems.
From a theoretical perspective, any criterion based on the nonlinear controllability matrix depends on exact knowledge of system parameters, which limits its applicability in real-world settings where parameters are uncertain or unavailable.

\subsection{Linear Structural Controllability}
Constraints of scalability and system identification are longstanding issues for verifying controllability, even for linear control systems of the form
\begin{equation}\label{eq: linear system}
    \dot{\textbf{x}}(t)= \textbf{A} \textbf{x}(t) + \textbf{B} \textbf{u}(t).
\end{equation}
These limitations stem from the fact that the set of controllable linear systems is open and dense with respect to the standard matrix norm, meaning that almost any small perturbation of system parameters can render a system controllable.

\begin{proposition}[\cite{lee1967foundations}]\label{thm: Av Bv is open and dense}
    If the linear control system \eqref{eq: linear system} is controllable, then there exists an $\varepsilon>0$ such that any perturbed system
    \begin{equation}\label{eq: perturb}
        \dot{\textbf{x}}(t)=\tilde{\textbf{A}}\textbf{x}(t)+\tilde{\textbf{B}}\textbf{u}(t),
    \end{equation}
    where $\|\tilde{\textbf{A}}-\textbf{A}\|<\varepsilon$ and $\|\tilde{\textbf{B}}-\textbf{B}\|<\varepsilon$, is also controllable. If the linear control system \eqref{eq: linear system} is not controllable, then for each $\varepsilon>0$ there exists a controllable system of the form \eqref{eq: perturb} within the same $\varepsilon$-neighborhood of \textbf{A} and \textbf{B}.
\end{proposition}

This result indicates that arbitrarily small perturbations to the system matrices can change the controllability of the system.
In practice, this implies verifying controllability of \eqref{eq: linear system} depends upon infinitely precise knowledge of the system parameters.
The sensitivity highlighted by Proposition \ref{thm: Av Bv is open and dense} motivated Lin’s concept of structural controllability, which explicitly accounts for the fact that ``only some of these entries are known with 100 percent precision'', referring to the zero entries of the system matrices \cite{lin1974structural}.
The sparsity patterns of the matrices $\textbf{A}$ and $\textbf{B}$ define a directed graph $\mathcal{G} = \{\mathcal{V}, \mathcal{E}\}$.
Each state  $x_j$ corresponds to a state node $v_j \in \mathcal{V}$, and each control input $u_j$ corresponds to a control node $v_{n+j} \in \Vs$.
A directed state edge $(v_j, v_i)$ exists if $\textbf{A}_{ij} \neq 0$, and a directed control edge $(v_{n+j}, v_i)$ exists if $\textbf{B}_{ij} \neq 0$.
Independent of the numerical values of the system parameters, this graph encodes the parameter support of \eqref{eq: linear system}.
The graph $\Gs$ is said to be structurally controllable if there exists at least one numerical realization of $\textbf{A}$ and $\textbf{B}$ consistent with the zero-nonzero pattern that renders the system controllable.

\begin{proposition}[\cite{lin1974structural}]\label{thm: lin structural controllability}
    Consider the linear control system \eqref{eq: linear system} with a single control input. The system is structurally controllable if and only if the associated graph $\mathcal{G}$ satisfies two conditions: (i) Accessibility: every state node $v \in \mathcal{V}$ is reachable from a directed path in $\mathcal{G}$ beginning at $v_{n+1}$; (ii) No dilations: for every subset of state node $\mathcal{V}' \subseteq \mathcal{V}$, the number of in-neighbors of $\mathcal{V}'$ is at least $|\mathcal{V}'|$ ($|\cdot|$ denotes set cardinality).
\end{proposition}

Only when both conditions are satisfied can the sparsity pattern of $(\textbf{A},\textbf{B})$ meet the exact controllability criterion according to the PBH test \cite{lin1974structural}.
The accessibility condition ensures that the control inputs can influence all state variables, while the absence of dilations prevents structural rank deficiencies in the controllability matrix.
Since both conditions depend solely on the graph structure, they enable scalable algorithms for verifying controllability.
In particular, \cref{thm: lin structural controllability} can be verified via  maximum matching algorithms on $\Gs$, which guides driver node selection in large-scale graphs \cite{liu2011controllability}.
Structural controllability provides a framework that is both computationally scalable and robust to uncertainties in system identification.
In the following section, we extend these concepts to polynomial dynamics and their associated hypergraph representations.

\section{Structural Controllability of Hypergraphs}\label{sec:3}

This section presents the main theoretical results extending structural controllability to homogeneous polynomial systems via hypergraph representations.
We first characterize identification limitations, then introduce the necessary hypergraph structures, develop a structural controllability framework, and finally present scalable verification algorithms.

\subsection{Limitations Imposed by System Identification}\label{sec: 3a}
Applying \cref{thm: jurdjevic and kupka} through rank verification of the nonlinear controllability matrix is fundamentally limited by the accuracy of system identification. In particular, an analogue of Proposition \ref{thm: Av Bv is open and dense} in linear systems also holds for tensor-based homogeneous polynomial systems. Consequently, controllability is not robustly distinguishable from uncontrollability under small parameter perturbations or uncertainty.
\begin{proposition}\label{thm: At Bv is open and dense}
    If the homogeneous polynomial system with linear input \eqref{eq: tensor with linear inputs} is controllable according to \cref{thm: jurdjevic and kupka}, then there exists an $\varepsilon>0$ such that every system
    \begin{equation}  \label{eq: similarpolynomial}     
         \dot{\textbf{x}}(t)=\tilde{\mathscr{A}}\textbf{x}(t)^{k-1}+\tilde{\textbf{B}}\textbf{u}(t),
    \end{equation}
    where $\|\tilde{\mathscr{A}}-\mathscr{A}\|<\varepsilon$ and $\|\tilde{\textbf{B}}-\textbf{B}\|<\varepsilon,$
    is also controllable. If the  system (\ref{eq: tensor with linear inputs}) is not controllable, then for each $\varepsilon>0$, there exists a controllable system of the form (\ref{eq: similarpolynomial}) within the same $\epsilon$-neighborhood of $\mathscr A$ and $\textbf B$.
\end{proposition}

\begin{proof}
    If the homogeneous polynomial system with linear input \eqref{eq: tensor with linear inputs} is controllable, the rows of the controllability matrix $\textbf{C}$ are linearly independent such that the matrix is full rank.
    If $\|\mathscr{A}-\tilde{\mathscr{A}}\|<\varepsilon$ and $\|\textbf{B}-\tilde{\textbf{B}}\|<\varepsilon$ for sufficiently small $\varepsilon>0,$ then the rows of the controllability matrix $\tilde{\textbf{C}}$ approximate the rows of $\textbf C$ and also be linearly independent. On the other hand, if system (\ref{eq: polynomial with linear inputs}) is not controllable, the tensor $\tilde{\mathscr{A}}$ and matrix $\tilde{\textbf{B}}$ can be chosen so that $\|\tilde{\mathscr{A}}-\mathscr{A}\|<\varepsilon$ and $\|\tilde{\textbf{B}}-\textbf{B}\|<\varepsilon,$ and all entries of $\tilde{\mathscr{A}}$ and $\tilde{\textbf{B}}$ are algebraically independent over the rational numbers. The algebraic independence indicates there are no nontrivial rational polynomials between the entries of $\tilde{\mathscr{A}}$ and $\tilde{\textbf{B}}$.
    The controllability matrix $\tilde{\textbf{C}}$ has no $n\times n$ submatrix with a determinant of zero because each determinant is a polynomial of the entries of $\tilde{\mathscr{A}}$ and $\tilde{\textbf{B}}$. Thus, $\text{rank}(\tilde{\textbf{C}})=n,$ and the pair $(\tilde{\mathscr{A}},\tilde{\textbf{B}})$ is controllable.
\end{proof}

The proof generalizes the argument of Lee and Markus for \cref{thm: Av Bv is open and dense} by replacing the linear controllability matrix with the nonlinear controllability matrix \cite{lee1967foundations}. This proposition implies that exact controllability is generically satisfied but structurally fragile with respect to parameter identification. Therefore, controllability verification based on precise tensor entries may not be reliable in data-driven settings. This requirement can be relaxed to consider only which entries of $\mathscr{A}$ are nonzero. Similar to the use of graphs to represent the sparsity structure of linear systems, the sparsity of polynomials can be represented with directed hypergraphs.

We encode the support patterns of the dynamic tensor $\mathscr{A}$ and input matrix $\textbf{B}$ as a directed hypergraph. Each nonzero coefficient corresponds to a hyperedge, with the tail listing the multiplicative state variables and the head indicating the affected state components.

\begin{definition}\label{def: HG of tA Bv}
Given the odd-degree homogeneous polynomial control system~\eqref{eq: tensor with linear inputs}, 
the associated directed hypergraph 
$\mathcal{H}(\mathscr{A},\mathbf{B}) = \{\mathcal{V}, \mathcal{E}\}$ 
is defined as follows. The node set is
\[
\mathcal{V} = \{v_1, v_2, \dots, v_n, v_{n+1}, \dots, v_{n+m}\},
\]
where $\{v_1, v_2, \dots, v_n\}$ represent the state nodes and 
$\{v_{n+1}, v_{n+2}, \dots, v_{n+m}\}$ represent the control nodes.   The hyperedges $\mathcal{E}$ are partitioned into state and control hyperedges. The state hyperedges are
    \begin{align*}       
        \mathcal{E}^\text{state}=&\Big\{\big(\{v_{h}\},\{v_{t_1}, v_{t_2},\dots,v_{t_{k-1}}\}\big)\;\big|\;\mathscr{A}_{h{t_1}{t_2}\dots {t_{k-1}}}\neq 0\Big\},
    \end{align*}
    where $e^h=\{v_{h}\}$ is the head and $e^t=\{v_{t_1},v_{t_2},\dots,v_{t_{k-1}}\}$ is the tail,
    while the control hyperedges are
    \begin{equation*}
        \mathcal{E}^{\text{control}}
        =
        \Big\{
        \big(
        v_{h},
        v_t\big)
        \;|\;
        \mathbf{B}_{h(t-n)} \neq 0
        \Big\}.
    \end{equation*}
    The full set of hyperedges is then $\mathcal{E}=\mathcal{E}^\text{state}\cup\mathcal{E}^\text{control}.$
\end{definition}

\Cref{def: HG of tA Bv} is the direct, multi-way generalization of the graph-based representation for a linear system initially developed structural controllability \cite[Section II]{lin1974structural}. For notational simplicity, each hyperedge is written with a single head. More generally, the head set $e^h$ may contain multiple nodes without affecting the subsequent structural analysis.
There are several elements in \Cref{def: HG of tA Bv} that are nonstandard for hypergraphs.
First, the hyperedge set may contain both pairwise and higher-order hyperedges.
Second, this definition allows for hyperedges to contain nodes multiple times.
For example, if the $\mathscr{A}_{122}\neq0$, i.e., $\dot{{x}}_1$ is influenced by ${x}_2^2$, then there will be a hyperedge written as $(\{v_1\},\{v_2,v_2\}).$
Third, the partition of each hyperedge into heads and tails does not specify that the head and tail are mutually exclusive.
These choices, however, allow for the most direct generalization of structural control from graphs to hypergraphs.
Equipped with the correspondence between the tensor-based homogeneous polynomial dynamics and a directed hypergraph, we define the structural controllability of \eqref{eq: tensor with linear inputs} accordingly.
We say two hypergraphs $\mathcal{H}(\mathscr{A},\textbf{B})$ and $\mathcal{H}(\tilde{\mathscr{A}},\tilde{\textbf{B}})$ are identical if they share the same node and hyperedge sets.
This occurs precisely when $\mathscr{A}$ and $\tilde{\mathscr{A}}$ have the same sparsity pattern, i.e.,
\[
\mathscr{A}_{j_1 j_2 \cdots j_k} = 0 \Longleftrightarrow \tilde{\mathscr{A}}_{j_1 j_2 \cdots j_k} = 0.
\]

\begin{definition}\label{def: structural controllability of polynomials}
An odd-degree homogeneous polynomial control system \eqref{eq: tensor with linear inputs} represented by $(\mathscr{A},\textbf{B})$ is structurally controllable if there exists a {controllable} realization $(\tilde{\mathscr{A}}, \tilde{\textbf{B}})$ whose associated hypergraph $\mathcal{H}(\tilde{\mathscr{A}},\tilde{\textbf{B}})$ is identical to $\mathcal{H}(\mathscr{A},\textbf{B})$.
\end{definition}

Because structurally equivalent pairs $(\mathscr{A},\textbf{B})$ and $(\tilde{\mathscr{A}}, \tilde{\textbf{B}})$ admit identical hypergraph representations, hypergraph-based algorithms can be used to determine whether there exists a realization of the system \eqref{eq: tensor with linear inputs} that satisfies the conditions of \cref{thm: jurdjevic and kupka} and is therefore strongly controllable.

\subsection{Hypergraph Structure for Polynomial Controllability}\label{sec: 3b}
To extend \cref{thm: lin structural controllability} to polynomial control systems via hypergraph representation, we introduce directed hypergraph walks, accessibility, and dilations, extending these concepts from the graph setting (see Fig. \ref{fig:1}).

\begin{figure*}[t]
    \centering
    \includegraphics[width=0.75\linewidth]{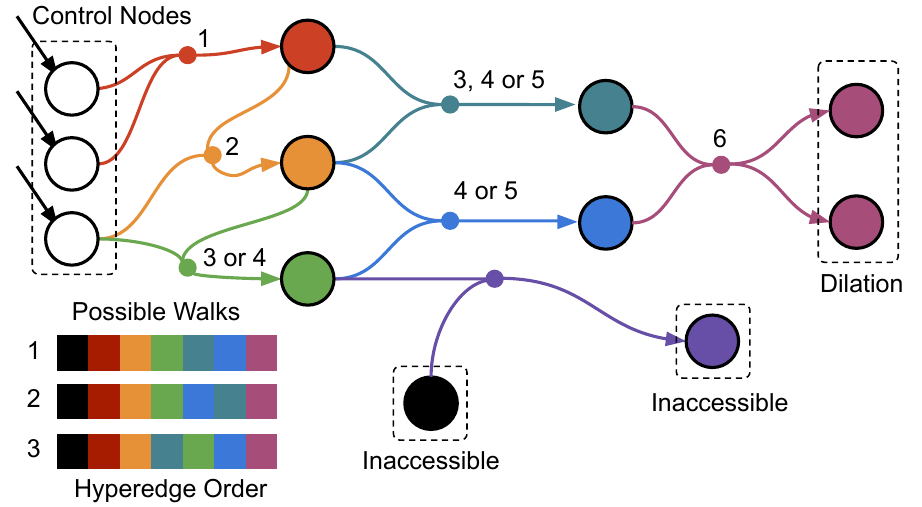}
    \caption{Examples of hypergraph walk, accessibility, and dilation. This hypergraph is not structurally controllable. Hyperedges are uniquely colored, and nodes are shaded to match the color of their incident hyperedge head. Hyperedges are numbered according to the order they may appear in a walk originating at the control node. The black and dark purple nodes are inaccessible, and the light purple nodes form a dilation. The black node is inaccessible as it is not an input node and lacks an incident hyperedge head. Consequently, the dark purple node is also inaccessible; its incident hyperedge head cannot be traversed because the black node is inaccessible. The light purple nodes form a dilation since there is only a single hyperedge head for two nodes.}
    \label{fig:1}
\end{figure*}

\subsubsection{Directed Hypergraph Walks}\label{sec: directed hypergraphs}
\Cref{def: HG of tA Bv} is designed to allow the system \eqref{eq: tensor with linear inputs} to be described in terms of the hyperedges.
Given a $k$-tail uniform hypergraph $\mathcal{H}(\mathscr{A},\textbf{B})$, the dynamics of \eqref{eq: tensor with linear inputs} can be written equivalently as
\begin{equation}\label{eq: hypergraph dynamics incidence perspective} 
    \dot{{x}}_j(t)=\sum_{e\in\mathcal{E}\text{ s.t. }v_j\in e^h}\quad\prod_{i\text{ s.t. }v_i\in e^t}{x}_i(t) +\textbf{b}_{j}\textbf{u}(t).
\end{equation}
This form is advantageous because it expresses the system \eqref{eq: tensor with linear inputs} in terms of the  hyperedge set, enabling its use as a computational graph.
The system \eqref{eq: hypergraph dynamics incidence perspective} allows the propagation of an input or signal on the hypergraph to be expressed as the following notion of a hypergraph walk.
\begin{definition}\label{def: walk} 
    A hypergraph walk starting from the sets of visited nodes $\{\{v_{n+1}\},\{v_{n+2}\},\dots,\{v_{n+m}\}\}$ is a sequence of hyperedges $w=(e_1,e_2,\dots,e_l)$ that satisfies the following conditions: 
    (i) The set of visited sets of nodes is maintained throughout the walk, beginning with         \[\mathcal{W}_0=\big\{\{v_{n+1}\},\{v_{n+2}\},\dots,\{v_{n+m}\}\big\};\] 
    (ii) Prior to the appearance of $e_j$ in the walk sequence, there exist sets 
    $h_1,h_2,\dots,h_p \in \mathcal{W}_{j-1}$ such that
    \[e_j^t \in \mathfrak{A}(h_1,h_2,\dots,h_p),\]
    where $\mathfrak{A}(h_1,h_2,\dots,h_p)$ denotes the set algebra generated by 
    $h_1,h_2,\dots,h_p$ under finite unions, intersections, and set differences;
       (iii) When a new hyperedge $e_j$ is added to the walk sequence $w$, nodes in the head $e_j^h$ are marked as visited nodes, and
       the sets of visited sets of nodes update as \[\mathcal{W}_j=\mathcal{W}_{j-1}\cup e_j^h.\]
    For a walk $w$ of length $l,$ the nodes in $e^h_l$ are the final nodes of the walk.
\end{definition}

This definition of a hypergraph walk is motivated by the need to provide a combinatorial process on a hypergraph that captures the tensor and Kronecker multiplications, which govern the dynamics of the system \eqref{eq: tensor with linear inputs}.
For a standard graph, suppose $\textbf x(0)$ is a vector where exactly one entry is nonzero. If $x_j(0)=1,$ then a walk begins on node $v_j.$ The multiplication $\textbf x(1)=\textbf{A}\textbf x(0)$ produces a vector $\textbf x(1)$ whose nonzero entries indicate the possible next node to include in the walk.
Analogously, for the proposed hypergraph walk, the dynamics \eqref{eq: hypergraph dynamics incidence perspective} implies that if $\textbf x(0)$ is defined as the initial nodes, the multiplication $\textbf x(1)=\mathscr{A}\textbf x(0)^{k-1}$ gives the set of nodes that could be included in the next step of the walk.

Two distinctions arise between \cref{def: walk} and a standard walk on a graph.
First, in the graph setting where a walk visits only one node at a time, the nonzero entries of $\textbf x(1)$ correspond to possible nodes that may be included in the walk, but only one node is included.
By contrast, in the hypergraph setting, the hyperedge heads $e^{h}$ may contain multiple nodes.
When a hyperedge is activated, all nodes in $e^{h}$ are visited simultaneously. As a result, the visited sets of nodes $\mathcal{W}_j$ may expand to multiple nodes in a single step, leading to simultaneous nonzero entries in $\textbf{x}(1)$.
For example, if there exists a hyperedge $\{\{v_1, v_2\}, \{v_3\}\},$ a walk beginning with $\mathcal{W}_0=\{v_3\}$ will visit $v_1$ and $v_2$ together.
We refer to a node $v$ as individually accessible if there exist a set of walks that isolate $v.$
Second, the standard graph multiplication $\textbf x(t+1)=\textbf{A}\textbf x(t)$ is linear, while the hypergraph multiplication $\textbf x(t+1)=\mathscr{A}\textbf x(t)^{k-1}$ is multilinear.
The tensor case is multilinear with respect to a single input argument $\textbf x(t),$ but $\mathscr{A}$ could also act on previous state vectors in the fashion of \[\textbf x(t+1)=\mathscr{A}\textbf x(t)^{k-r-1}\textbf x(t-1)^r\] for $0<r<k-1.$
In this case, the nonzero entries of $\textbf x(t+1)$ are a function of both the hypergraph structure defined by $\mathscr{A}$ and vectors $\textbf x(t)$ and $\textbf x(t-1),$ whose sparsity correspond to nodes visited at the current and previous time points.
This is extensible beyond only the immediately previous time point to include $k-1$ different sets of nodes that have been visited.
For these reasons, a hypergraph walk defined in \cref{def: walk} allows groups of nodes to be visited simultaneously and requires all nodes in $e^t$ to be visited before $e$ is included in the walk.

\subsubsection{Accessibility and Dilations}\label{sec: structural properties}
To extend the original structural controllability of linear systems, we require the following definitions of dilations node accessibility.

\begin{definition}\label{def: hyperedge dilation}
Given a directed  hypergraph $\Hs=\{\Vs,\Es\}$, 
a set of nodes $\Ss\subseteq\Vs$ is a dilation if the number of distinct sets of the form $\Ss\cap e^h$, taken over all hyperedges $e\in\Es$ whose head intersects with $\Ss$, is strictly less than $|\Ss|$.

\end{definition}

\begin{definition}
    Given a directed  hypergraph $\Hs=\{\Vs,\Es\}$ and a set of nodes $\Ws_0=\{\{v_{n+1}\},\{v_{n+2}\},\dots,\{v_{n+m}\}\},$ a node is accessible if it is visited during a walk beginning at $\Ws_0.$
\end{definition}

Both definitions mirror their graph counterparts with slight modifications.
The hypergraph dilation differs from a graph dilation defined by Lin, which uses the size of $\Ss$'s node neighborhood (adjacent nodes) rather than the number of hyperedges directed into $\Ss$.
In a standard graph, where each edge has a single tail, these two quantities coincide, so the distinction becomes meaningful only in the hypergraph setting.
The reason that \cref{def: hyperedge dilation} considers the incident hyperedges rather than adjacent nodes is that the allowance of multiple nodes with each hyperedge increases the number of input signals into $\Ss.$
As an example, suppose that $|\Ss|=5,$ and all hyperedges whose heads included elements of $\Ss$ contained at most 4 nodes in their tails.
This would still allow for $\sum_{j=1}^4\binom{4}{j}=15$ unique pathways to regulate the 5 elements of $\Ss$.
Separately, the hyperedge heads must be unique with respect to the nodes in $\Ss.$
For instance, let $\Ss=\{v_1,v_2\}$ and suppose that the only three hyperedges influencing nodes in $\Ss$ are $\{\{v_1,v_2,v_6\},\{v_3\}\}$, $\{\{v_1,v_2,v_7\},\{v_4\}\}$, and $\{\{v_1,v_2,v_8\},\{v_5\}\}.$ 
Then a hyperedge dilation exists, since the set of hyperedge heads unique with respect to nodes in $\Ss$ is one.
Equivalently, any information passed to $v_1$ is also passed to $v_2$ and vice versa.

\subsection{Structural Controllability of Polynomial Systems}\label{sec: 3c}
We now utilize hypergraph structure to determine conditions under which $(\mathscr A,\textbf B)$  possesses the sparsity required to satisfy the strong controllability criterion of \cref{thm: cohg}.
We first consider the sparsity of a matrix $\textbf W$ generated from a hypergraph walk originating at the control nodes.
Beginning with control nodes $\Ws_0=\{\{v_{n+1}\},\{v_{n+2}\},\dots,\{v_{n+m}\}\}$ allows the first step of the walk to progress from any control nodes to any state nodes that are influenced by the columns of $\textbf{B}.$ 
The second step of this walk may then visit any nodes corresponding to nonzero entries in $\textbf{w}$ defined as
\begin{equation*}
    \textbf{w}=\mathscr A\times_2\textbf{b}_{j_1}\times_3\textbf{b}_{j_2}\times_4\dots\times_{k}\textbf{b}_{j_{k-1}},
\end{equation*}
where $j_1,j_2,\dots,j_{k-1}$ denote the column indices of $\textbf B.$
Note that the possible sparsity of $\textbf w$  is equivalent to the possible sparsity of $\textbf M_1$ in \cref{thm: cohg}.
This tensor notation can be rewritten in terms of the Kronecker product as
\begin{equation*}
    \textbf W_1=\textbf A_{(1)}\textbf B^{[k-1]},
\end{equation*}
where the columns of $\textbf W_1$ denote possible next sets of nodes for the walk.
According to the definition of the walk, the matrix $\textbf W_2$ is constructed as
\begin{equation*}
    \textbf W_2=\textbf A_{(1)}\begin{bmatrix}
        \textbf B&\textbf W_1
    \end{bmatrix}^{[k-1]},
\end{equation*}
 where the sparsity structure of columns of $\textbf W_2$ denote the set of nodes accessible after this next step.
We continue this pattern and denote  $\textbf W_j$ as the matrix where the sparsity structure of each column indicates a set of nodes accessible at the $j$th step.
Note that the columns sparsity of $\textbf W_j$ is equivalent to the columns sparsity of $\textbf M_j,$ so that the matrix
\begin{equation*}
    \textbf W=\begin{bmatrix}
    \textbf B&\textbf W_1&\textbf W_2&\cdots&\textbf W_{n}
\end{bmatrix}
\end{equation*}
has an identical sparsity to the nonlinear controllability matrix \textbf{C}.
If entries of $\textbf W$ are selected to be full rank, then according to \cref{thm: At Bv is open and dense}, there exist a strongly controllable pair $(\tilde{\mathscr A},\tilde{\textbf B})$ with equivalent sparsity as $(\mathscr A,\textbf B)$, and we call the system structurally controllable.

\begin{proposition}\label{thm: structural control of polynomials}
   The system $(\mathscr A,\textbf B)$ is structurally controllable if and only if the hypergraph $\Hs(\mathscr A,\textbf B)$ contains neither a hyperedge dilation nor an inaccessible node.
\end{proposition}

\begin{proof}
    We first show the sufficiency.
    If $(\mathscr A,\textbf B)$ contains a nonaccessible node $v$, then there exists no walk beginning at the control nodes that will reach node $v$.
    This indicates there will be a row of all zeros in the matrix $\textbf W_j$ for all possible steps $j$ in the walk.
    There will then be a row of all zeros in $\textbf W.$
    Since $\textbf W$ has the same sparsity structure as the nonlinear controllability matrix, then for any pair $(\tilde{\mathscr A},\tilde{\textbf B})$ the nonlinear controllability matrix will have a row of all zeros and be low rank.
    Therefore, the pair $(\mathscr A,\textbf B)$ is not structurally controllable.
If $\Hs(\mathscr A,\textbf B)$ contains a dilation, the augmented matrix $\textbf R=\begin{bmatrix}
        \tilde{\textbf A}_{(1)}&\tilde{\textbf B}
    \end{bmatrix}$ may be partitioned as
    \begin{equation*}
        \textbf R=\begin{bmatrix}
            \tilde{\textbf A}_{(1)}&\tilde{\textbf B}    \end{bmatrix}=\begin{bmatrix}
            \textbf D_\mathscr A&\textbf D_{\textbf B}\\ \textbf N_\mathscr A&\textbf N_\textbf B
        \end{bmatrix},
    \end{equation*}
    where $\textbf D=\begin{bmatrix}
        \textbf D_\mathscr A&\textbf D_\textbf B
    \end{bmatrix}$ is a $d\times (n^{k-1}+m)$ matrix whose rows correspond to nodes in the dilation set $\Ss$, and $\textbf N=\begin{bmatrix}
        \textbf N_\mathscr A&\textbf N_\textbf B
    \end{bmatrix}$ is a $(n-d)\times (n^{k-1}+m)$ matrix whose rows correspond to nodes not found within the dilation.
    Since $\Ss$ is a dilated set of nodes, there exists $d-1$ or fewer hyperedges whose heads are unique with respect to the elements in $\Ss.$
    As the rows of $\textbf D$ are nodes and the columns are hyperedges, there exists at most $d-1$ unique column vectors in $\textbf D,$ so $\rank(\textbf D)<d.$
    The matrix $\textbf W$ has the sparsity structure of
    \begin{equation*}
    \begin{split}
        \textbf W&=\begin{bmatrix}
            \textbf D_\textbf B 
        & \textbf D_\mathscr A\textbf W_0^{[k-1]}&\cdots&\textbf D_\mathscr A\textbf W_{n-1}^{[k-1]}\\
            &&\\
            \textbf N_\textbf B & \textbf N_\mathscr A \textbf W_0^{[k-1]}&\cdots&\textbf N_\mathscr A\textbf W_{n-1}^{[k-1]}\\
        \end{bmatrix}.
    \end{split}
    \end{equation*}
    The rank of $\textbf W$ is bound by the relationship
    \begin{equation*}
    \rank(\begin{bmatrix}
            \textbf D_\mathscr A\!&\!\textbf D_\textbf B
    \end{bmatrix})\!\geq\!    \rank(\begin{bmatrix}\textbf D_\textbf B \!&\! \textbf D_\mathscr A\textbf W_0^{[k-1]} \!&\!\cdots\!&\!\textbf D_\mathscr A\textbf W_{n-1}^{[k-1]}\end{bmatrix}).
    \end{equation*}
    Hence, the upper block of $\textbf W$ has a rank less than $d.$ As a result, $\rank(\textbf W)<n,$ and the nonlinear controllability matrix will also be low rank since it has an identical sparsity structure. Therefore, the pair $(\mathscr A,\textbf B)$ is not structurally controllable.

    We have shown the presence of either a dilation or inaccessible node is sufficient to conclude that the system is not structurally controllable. We now show the converse direction. If $(\mathscr A,\textbf B)$ is not structurally controllable, then $\rank(\textbf C) <n.$
    The nonlinear controllability matrix may be low rank when the augmented matrix $\textbf R=\begin{bmatrix}
        \tilde{\textbf A}_{(1)}&\tilde{\textbf B}
    \end{bmatrix}$ has either $\rank(\textbf R)=n$ or $\rank(\textbf R)<n.$
    In the case where the augmented matrix $\textbf R$ has a rank less than $n,$ there must exist at least one row of $\textbf R$ that has all zero entries, or $n^{k-1}+m-n+1$ columns that has all zero entries. To satisfy the condition that $\rank(\textbf R)<n$, these rows and columns contain only zero entries, rather than having a nontrivial linear dependence, because if there was a nontrivial linear dependence and there were some nonzero entries, values could be chosen in a fashion consistent with the proof provided by Lee and Markus to increase the rank of $\textbf R.$ Given that these entries are fixed as zeros, the following structures occur in $\Hs(\mathscr A,\textbf B).$ A row of all zeros corresponds to a node that receives no inputs. As a result, the node corresponding to this row is not contained in any hyperedge heads, so there is no walk from the control nodes to this node. Therefore, a nonaccessible node is found. Moreover, at least $n^{k-1}+m-n+1$ columns containing all zeros indicates at most $n-1$ hyperedges in $\Hs(\mathscr A,\textbf B).$ Then, the full set of $n$ nodes of $\Hs(\mathscr A,\textbf B)$ constitutes a dilation since $|\Vs|>|\Es|.$
    It is possible that both a nonaccessible node and dilation exist so that $\rank(\textbf R)<n.$

    In the case where the augmented matrix $\Rv$ has rank $n$ and $\rank(\textbf C)<n$, it will be the case that $\rank(\tilde{\textbf A}_{(1)}\tilde{\textbf B}^{[k-1]})<n,$ since this matrix is a submatrix of the nonlinear controllability matrix.
    Although $\tilde{\textbf A}_{(1)}$ itself may have rank $n$, its image is constrained when acting on inputs from the column space of $\tilde{\textbf B}^{[k-1]}.$ This may occur for two reasons. First, at least one row of $\tilde{\textbf A}_{(1)}\tilde{\textbf B}^{[k-1]}$ may be all zeros. In this case, there are nodes that are nonaccessible on $\Hs(\At,\Bv)$ from a length 1 walk beginning at the control nodes. Second, at least two rows of $\tilde{\textbf A}_{(1)}\tilde{\textbf B}^{[k-1]}$ have a nontrivial linear dependence. If this is the case, there exists a permutation of the coordinates, such that the matrix $\tilde{\textbf A}_{(1)}\tilde{\textbf B}^{[k-1]}$ may be written
        \begin{equation*}
            \tilde{\textbf A}_{(1)}\tilde{\textbf B}^{[k-1]}=\begin{bmatrix}
                \textbf D\\
                \textbf N
            \end{bmatrix},
        \end{equation*}
        where $\textbf D$ denotes the rows that have a nontrivial linear dependence, and $\Nv$ denotes the rows that are linearly independent. Here, $\textbf D$ is a $d\times m^{k}$ matrix, and $\textbf N$ is a $(n-d)\times m^{k}$ matrix, where $\rank(\textbf N)=n-d,$ and $\rank(\textbf D)<d.$ It is possible that $\textbf D$ contains a row of all zeros, in which case a nonaccessible node exist.
        Instead, if all the rows of $\textbf D$ contain a nonzero entry, there must exist at most $d-1$ linearly independent columns of $\textbf D$ to ensure that $\rank(\textbf D)<d.$
        Given that $\textbf D$ is constructed from matrices $(\tilde{\textbf A}_{(1)}, \tilde{\textbf B})$ where the structure is defined but the nonzero entries may take any value, the only way that at most $d-1$ linearly dependent columns are in $\textbf D$ is for $\textbf D$ to contain exactly $d-1$ unique columns, where the columns are unique with respect to which hyperedge tail nodes they are influenced by. As an example, in the tensor of a hypergraph containing a single hyperedge $\{\{v_1\},\{v_2,v_3\}\},$ the tensor structure indicates that the entries $\mathscr A_{123}$ and $\At_{132}$ correspond to the same hyperedge. When unfolded to the $3\times 9$ matrix $\textbf A_{(1)},$ columns 6 and 8 both describe the effects of this hyperedge.

        As the columns of $\textbf D$ indicate nodes contained in hyperedge heads, the set of $d$ nodes whose rows are described by $\textbf D$ has fewer than $d$ hyperedge heads directed at these nodes, so a dilation exists. Recall that $\textbf W_1=\tilde{\textbf A}_{(1)}\tilde{\textbf B}^{[k-1]}$, and it denotes the accessibility of nodes after a one step walk from the control nodes.
    From a walk starting at the sets of nodes indicated by the columns of $\textbf W_j,$ where $\textbf W_0=\Bv,$ the above process illustrates that ensuring $\rank(\textbf W_{j+1})<n$ requires the existence of a dilation or a nonaccessible node.
    Since a nonaccessible node from a walk beginning at $\textbf W_j$ is also nonaccessible from a walk beginning at $\textbf B,$ each step in the walk requires the existence of a dilation of a nonaccessible node.
    Thus, if the system is not structurally controllable then either a dilation or inaccessible node exists.
\end{proof}

In the proof, sufficiency follows directly, while necessity is established via the contrapositive. Consequently, \cref{thm: structural control of polynomials} implies that verifying the accessibility of all nodes and the absence of hyperedge dilations in the associated hypergraph provides a complete combinatorial test for structural controllability. This directly generalizes the classical structural controllability of linear systems on graphs, where controllability is determined by node accessibility and the absence of dilations. By formulating the problem in terms of hypergraph properties, the framework enables scalable verification for large-scale polynomial systems, facilitates the identification of critical nodes or hyperedges, and guides the design of input configurations that ensure controllability even under uncertainty in system parameters.

\subsection{Identification of Minimum Number of Driver Nodes}\label{sec: 3d}

Dilation detection algorithms can be used to efficiently compute lower bounds on the number of state variables that must receive control inputs in order to render a polynomial system controllable based on \cref{thm: structural control of polynomials}. By identifying subsets of nodes whose incoming hyperedges are insufficient to independently influence all states in the subset, these algorithms reveal structural limitations that must be resolved through additional control inputs. Consequently, the detected dilations indicate where additional driver nodes must be introduced to eliminate these structural bottlenecks and ensure that all states can be independently influenced.

\subsubsection{Hypergraph Matching-Based Lower Bound}
In the linear setting, Liu et al. employed minimal input theorem and maximum matching to identify the set of driver nodes \cite{liu2011controllability}. Graph matching algorithms are scalable and well suited for detecting dilations in both graphs and hypergraphs. To extend dilation detection to hypergraphs, we employ the star expansion, which converts a hypergraph into a bipartite graph whose node sets correspond to the nodes and hyperedges of the original hypergraph. For a standard hypergraph $\Hs(\Vs, \Es)$, the star expanded graph is defined as $\Gs(\Vs\cup\Es,\Es')$ where $\Vs\cup\Es$ is the augmented node set and the edge set is
\[
\Es' = \{ (v,e) \in \Vs \times \Es \mid v \in e, v\in \Vs, \text{ and }e\in\Es \}.
\]
A node $v \in \Vs$ is adjacent to a hyperedge node $e \in \Es$ in $\Gs$ if and only if $v$ is contained in the hyperedge $e$ in $\Hs$. We extend this definition to directed hypergraphs.

\begin{definition}
Given a hypergraph $\Hs=\{\Vs, \Es\}$ with node set $\Vs$ and hyperedge set $\Es$, the star expansion of $\Hs$, denoted by $\Gs = \mathrm{star}(\Hs)$, is the bipartite graph
    $\Gs = \{\Vs \cup \Es, \Es'\},$
where the edge set $\Es'$ is defined as
\begin{equation*}
\begin{split}
    \Es' = & \{ (v,e) \in \Vs \times \Es \mid v \in e^t, v\in \Vs, \text{ and }e\in\Es \} \cup\\
    & \{ (e,v) \in \Vs \times \Es \mid v \in e^h, v\in \Vs, \text{ and }e\in\Es \}.
\end{split}
\end{equation*}
\text{ }
\end{definition}

In the star expansion of a directed hypergraph, directed edges are assigned depending on whether each node $v$ belongs to the head or the tail of the hyperedge node $e$.
As an example, see Fig. \ref{fig:1}, where the state nodes $\Vs$ are large nodes and the nodes representing hyperedges are seen as smaller nodes. A matching in a directed graph is a set of edges such that no two selected edges share a common node. A node is said to be covered if it is the endpoint of an edge in the matching. The size of a matching is the number of covered nodes.
Based on \cref{def: hyperedge dilation}, a dilation exists in a hypergraph $\Hs=\{\Vs, \Es\}$ if there is no matching in the star-expanded graph $\Gs$ from hyperedge nodes to node nodes (i.e., $\Es \to \Vs$) that covers every node in $\Vs$. The presence of such a dilation implies that not all state variables can be independently influenced by the system dynamics.
This observation leads directly to a lower bound on the number of driver nodes required for controllability.

\begin{proposition}\label{thm: lower bound}
Given an uncontrolled polynomial system \eqref{eq: tensor with linear inputs}, its associated hypergraph representation is $\Hs$, and its star expansion is $\Gs=\text{star}(\Hs).$ If $\Gs$ admits a maximum matching of size $m$, then the number of uncovered nodes in $\Vs$, given by $|\Vs| - m$, provides a lower bound on the number of state variables that must be individually actuated in order to render the system controllable.
\end{proposition}
\begin{proof}
By construction, hyperedge dilations in $\Hs(\At)$ correspond to the absence of a matching in $\Gs$ from hyperedge nodes to node nodes that covers all nodes in $\Vs$.
Let $s$ be the size of a maximum matching in $\Gs$. Then exactly $|\Vs| - s$ nodes in $\Vs$ remain uncovered by any such matching. Each uncovered node corresponds to a state variable that cannot be matched to an influencing hyperedge and participates in a hyperedge dilation.
To eliminate all dilations, and thereby satisfy the conditions for structural controllability of \cref{thm: structural control of polynomials}, each uncovered node must be directly actuated by an independent control input. Therefore, at least $|\Vs| - s$ state variables must receive control inputs.
\end{proof}

The lower bound in \cref{thm: lower bound} follows from the necessity of eliminating hyperedge dilations for structural controllability. However, \cref{thm: structural control of polynomials} additionally requires the absence of inaccessible nodes. Since the matching-based procedure on the star expansion detects only hyperedge dilations and does not account for accessibility, it does not in general identify a sufficient set of driver nodes. Consequently, \cref{thm: lower bound} provides only a lower bound on the number of state variables that must be individually actuated.
This result offers a principled and computationally efficient starting point for driver node selection. Maximum matching on the star-expanded graph can be computed using scalable algorithms, making the bound practical for large-scale polynomial systems and a foundation for the driver selection strategies.

\begin{algorithm}[t]
\caption{Matching-Augmented Greedy (MaG) Driver Node Selection}
\label{alg:MaG}
\begin{algorithmic}[1]
\Require Homogeneous polynomial system with associated hypergraph representation $\Hs$
\Ensure Driver node set $\Ds$ satisfying structural controllability

\State Build the star expansion $\Gs = \mathrm{star}(\Hs)$
\State Compute a maximum matching $\mathcal{M}$ in $\Gs$
\State Identify unmatched nodes:
\[
\Ds_{\mathrm{dil}} = \{ v \in \Vs \mid v \text{ is uncovered by } \mathcal{M} \}
\]

\State Initialize driver set $\Ds \gets \Ds_{\mathrm{dil}}$
\State Compute accessible nodes $\As \gets \mathrm{WalkReach}(\Hs, \Ds)$

\While{$\As \neq \Vs$}
    \State Select $v \in \Vs \setminus \As$ that maximizes
    \[
    |\mathrm{WalkReach}(\Hs, \Ds \cup \{v\})|
    \]
    \State $\Ds \gets \Ds \cup \{v\}$
    \State Update accessible nodes $\As \gets \mathrm{WalkReach}(\Hs, \Ds)$
\EndWhile

\State \textbf{return} $\Ds$
\State

\Procedure{WalkReach}{$\Hs, \Ds$}
    \State Initialize accessible nodes $\As \gets \Ds$
    \State Initialize hyperedge queue $\Qs$ with hyperedges reachable from $\As$
    \While{$\Qs \neq \emptyset$}
        \State Pop hyperedge $e$ from $\Qs$
        \State Add head nodes to accessible set: $\As \gets \As \cup e^h$
        \State Update $\Qs$ with hyperedges now reachable from $\As$
    \EndWhile
    \State \Return $\As$
\EndProcedure

\end{algorithmic}
\end{algorithm}

\begin{table*}
\caption{Driver node selection for small-scale hypergraphs (mean $\pm$ std with 5 trials).}
\label{tab: exp 1}
\small
\centering
\setlength{\tabcolsep}{15pt}
\renewcommand{\arraystretch}{1.2}
\begin{tabular}{ll|ccc|c}
\toprule
$\alpha$ & $n$ & Matching & MaG & Greedy & Optimal \\
\midrule
\multicolumn{6}{c}{$k=6$}\\
\midrule
\multirow[c]{3}{*}{0.5} & 10 & 5.200 $\pm$ 0.447 & 7.200 $\pm$ 0.837 & 9.667 $\pm$ 0.577 & 6.800 $\pm$ 0.837 \\
 & 15 & 8.200 $\pm$ 0.447 & 10.200 $\pm$ 1.304 & 11.667 $\pm$ 0.577 & 9.200 $\pm$ 0.837 \\
 & 20 & 10.800 $\pm$ 0.837 & 13.200 $\pm$ 1.304 & 16.333 $\pm$ 1.528 & 12.400 $\pm$ 0.548 \\
\midrule
\multirow[c]{3}{*}{1.0} & 10 & 1.400 $\pm$ 0.894 & 7.400 $\pm$ 0.894 & 7.500 $\pm$ 0.577 & 6.000 $\pm$ 0.000 \\
 & 15 & 1.600 $\pm$ 0.548 & 10.000 $\pm$ 1.000 & 9.400 $\pm$ 0.548 & 7.000 $\pm$ 0.707 \\
 & 20 & 2.800 $\pm$ 1.304 & 11.800 $\pm$ 1.095 & 11.750 $\pm$ 1.500 & 8.800 $\pm$ 0.837 \\
\midrule
\multicolumn{6}{c}{$k=8$}\\
\midrule
\multirow[c]{3}{*}{0.5} & 10 & 5.600 $\pm$ 0.548 & 8.000 $\pm$ 0.000 & 8.200 $\pm$ 0.447 & 8.000 $\pm$ 0.000 \\
 & 15 & 8.400 $\pm$ 0.894 & 11.600 $\pm$ 1.140 & 13.000 $\pm$ 0.816 & 10.600 $\pm$ 0.894 \\
 & 20 & 10.600 $\pm$ 0.894 & 15.400 $\pm$ 0.548 & 17.400 $\pm$ 0.548 & 12.800 $\pm$ 1.483 \\
\midrule
\multirow[c]{3}{*}{1.0} & 10 & 2.400 $\pm$ 0.894 & 8.000 $\pm$ 0.000 & 9.000 $\pm$ 0.707 & 8.000 $\pm$ 0.000 \\
 & 15 & 2.000 $\pm$ 1.000 & 11.800 $\pm$ 1.304 & 11.600 $\pm$ 1.140 & 8.600 $\pm$ 0.548 \\
 & 20 & 3.200 $\pm$ 0.837 & 14.800 $\pm$ 1.643 & 15.600 $\pm$ 1.140 & 10.800 $\pm$ 0.447 \\
\bottomrule
\end{tabular}
\end{table*}

\subsubsection{Matching-Augmented Driver Node Selection}

The lower bound of \cref{thm: lower bound} motivates a scalable approach for driver node selection for hypergraph dynamics.
While the star-expansion matching detects only hyperedge dilations and does not account for accessibility, it identifies a minimal set of state variables that must be directly actuated.
These variables form a core set of driver nodes, which can then be augmented to ensure structural controllability.
Specifically, the proposed procedure constructs the hypergraph associated with a given homogeneous polynomial system, computes its star expansion, and determines a maximum matching. The unmatched nodes in this matching constitute the set
\[
\Ds_{\mathrm{dil}} = \{ v \in \Vs \mid v \text{ is uncovered by the maximum matching} \},
\]
which corresponds to the minimal set of nodes that eliminate all hyperedge dilations.
Determining $\Ds_{\mathrm{dil}}$ can be performed efficiently using classical bipartite matching algorithms, such as the Hopcroft–Karp algorithm and the Hungarian method \cite{hopcroft1973n, kuhn1955hungarian}.
Assigning independent control inputs to these nodes satisfies the lower bound of \cref{thm: lower bound}.

To achieve complete structural controllability, additional driver nodes may be required to ensure that all nodes are accessible from the selected drivers. Accessibility is defined in terms of directed hypergraph walks: a hyperedge can be traversed only after all nodes in its tail have been visited, and traversal of a hyperedge simultaneously activates all nodes in its head. This induces a forward-propagation process on the hypergraph, generalizing breadth-first search from graphs to directed hypergraphs (see \cref{alg:MaG}).
Starting from the initial set \(\Ds_{\mathrm{dil}}\), we compute the set of nodes reachable via hypergraph walks. Any remaining inaccessible nodes are then added iteratively using a greedy strategy: at each iteration, the node whose addition maximally increases the set of accessible nodes is selected as a driver node. This process continues until every node is reachable from the driver set. The resulting driver node set
\[
\Ds = \Ds_{\mathrm{dil}} \cup \Ds_{\mathrm{acc}}
\]
satisfies both the dilation-free and accessibility requirements of \cref{thm: structural control of polynomials}, guaranteeing structural controllability.

Previous work on hypergraph controllability employed a greedy selection heuristic that required \(\Os(n)\) computations of the singular value decomposition and Kronecker exponentiation per iteration \cite{chen2021controllability}. In contrast, our greedy selection leverages hypergraph walks with sparse representations, low memory usage, and intermediate results that can be cached between iterations, making it significantly more scalable while preserving the correctness guarantees provided by the graph-theoretic interpretation.
This matching-based identification of $\Ds_{\mathrm{dil}}$ followed by greedy accessibility completion yields a principled, scalable algorithm for driver node selection. Maximum matching can be computed efficiently, and the accessibility verification and augmentation admit efficient hypergraph-based implementations. Since the resulting procedure combines the lower-bound guarantee of maximum matching with a scalable, hypergraph-based greedy augmentation, we refer to it as matching-augmented greedy (MaG), which will be employed in subsequent numerical experiments to select driver nodes in large-scale hypergraph systems.

\section{Numerical Experiments}
\label{sec:numerical_experiments}

We present numerical experiments illustrating the application of structural controllability to large-scale systems.
Three classes of experiments are considered: (i) benchmarking the MaG approach on small-scale hypergraphs against alternative driver node selection methods; (ii) evaluating the scalability of MaG for high-dimensional hypergraphs; and (iii) analyzing driver node placement in systems with structured hypergraph topologies.  
All experiments were conducted on an AMD EPYC 7763 64-Core processor using a Python implementation of the Hypergraph Analysis Toolbox \cite{pickard2023hat}. The code for  these experiments
is available at \url{https://github.com/Jpickard1/hypergraph-structural-control}.

\subsection{Small-Scale Hypergraphs}
We first evaluate the performance of MaG on small-scale hypergraphs by comparing it with three alternative driver node selection methods:
(i) Optimal: enumerates all driver node combinations to identify the minimum set that satisfies structural controllability;
(ii) Maximum matching: selects driver nodes based solely on hypergraph matching;
(iii) Greedy selection: sequentially adds nodes to maximize the number of accessible nodes.
Due to computational complexity of a brute force search (BFS), we consider hypergraphs with $n \leq 20$ and $k \leq 8$, where $m$ is the total number of hyperedges. Random hypergraphs were sampled with densities $\alpha = m/n \in \{0.5, 1.0\}$. Summary results are provided in \cref{tab: exp 1}, and Fig. \ref{fig: exp 1} illustrates the representative case $k=4$. Consistent with \cref{thm: lower bound}, maximum matching alone fails to select a sufficient set of driver nodes to render the system structurally controllable. In contrast, greedy selection tends to choose the largest number of driver nodes. To quantify performance, we measure the optimality gap, defined as the difference between the number of driver nodes selected and the BFS optimum. Across all scenarios, MaG consistently achieves near-optimal performance while guaranteeing structural controllability. Although BFS becomes computationally prohibitive for larger systems, MaG remains computationally efficient and ensures that the selected driver nodes satisfy the structural controllability condition established in \cref{thm: structural control of polynomials}.

\begin{figure}[t!]
    \centering
    \includegraphics[width=\linewidth]{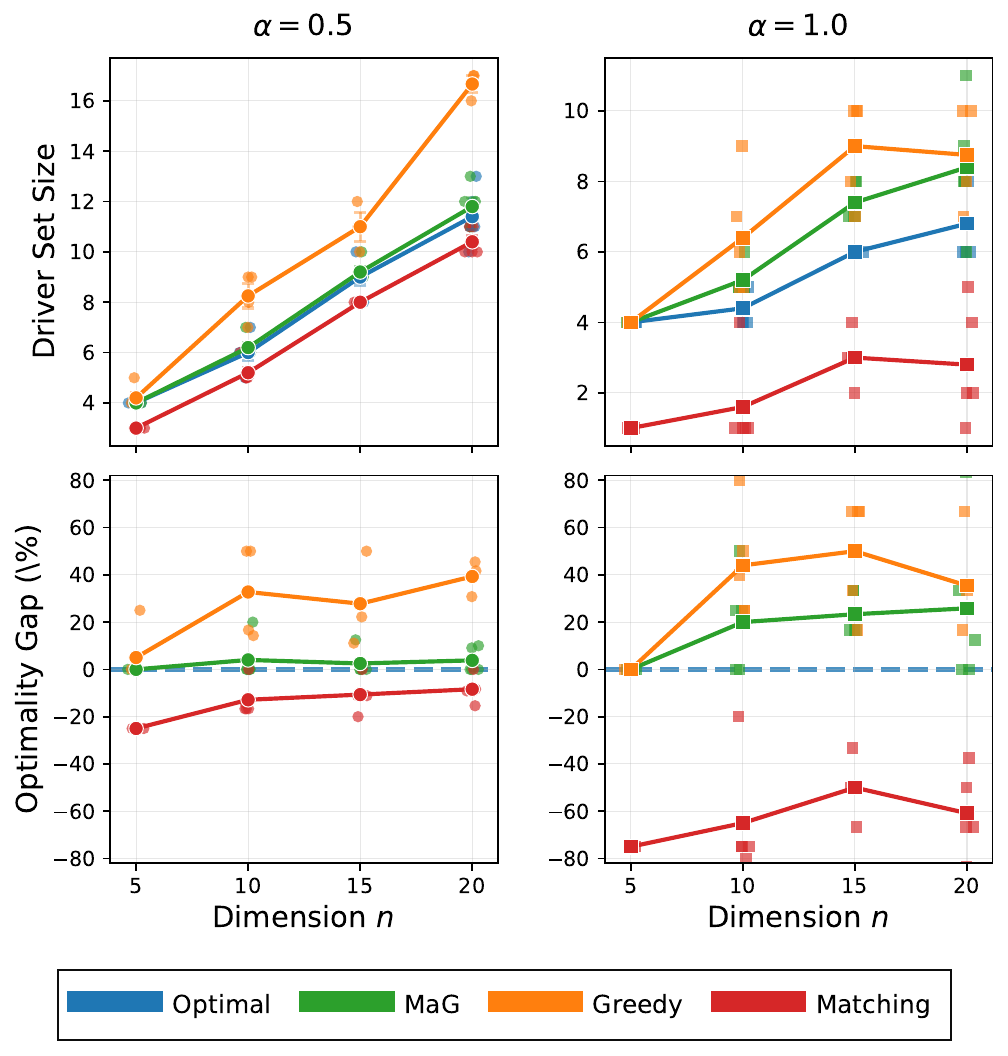}
    \caption{Driver node selection on small-scale hypergraphs. Comparison of four methods for selecting driver nodes in 4-uniform hypergraphs of varying size $n$ and density $\alpha$. Our proposed method (MaG), pure greedy selection, linear matching, and brute force search (optimal).}
    \label{fig: exp 1}
\end{figure}

\begin{figure}[t]
    \centering
    \includegraphics[width=\linewidth]{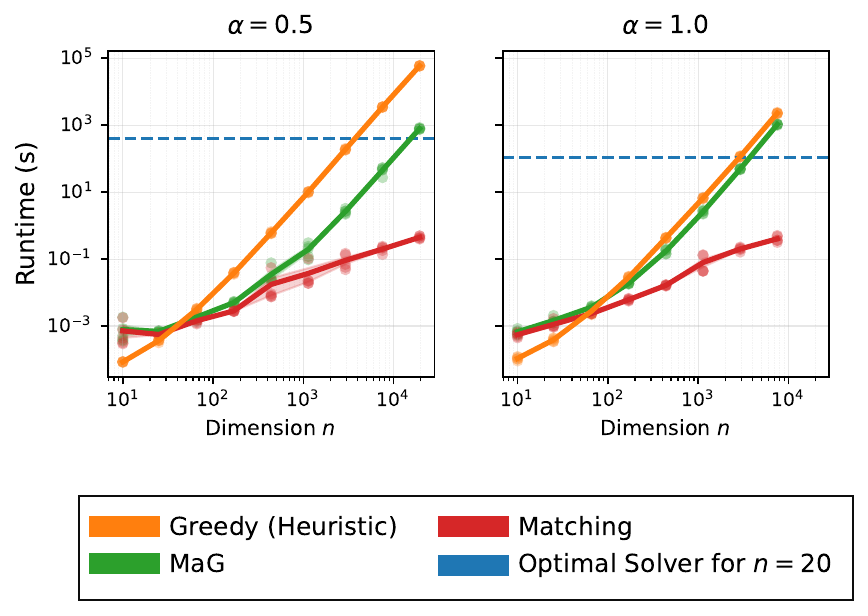}
    \caption{Scaling of MaG, matching, and greedy based control node selection for large-scale hypergraphs. Runtime is averaged over three trials per hypergraph size.}
    \label{fig: EXP 2}
\end{figure}

\begin{figure*}[t]
    \centering
    \includegraphics[width=\linewidth]{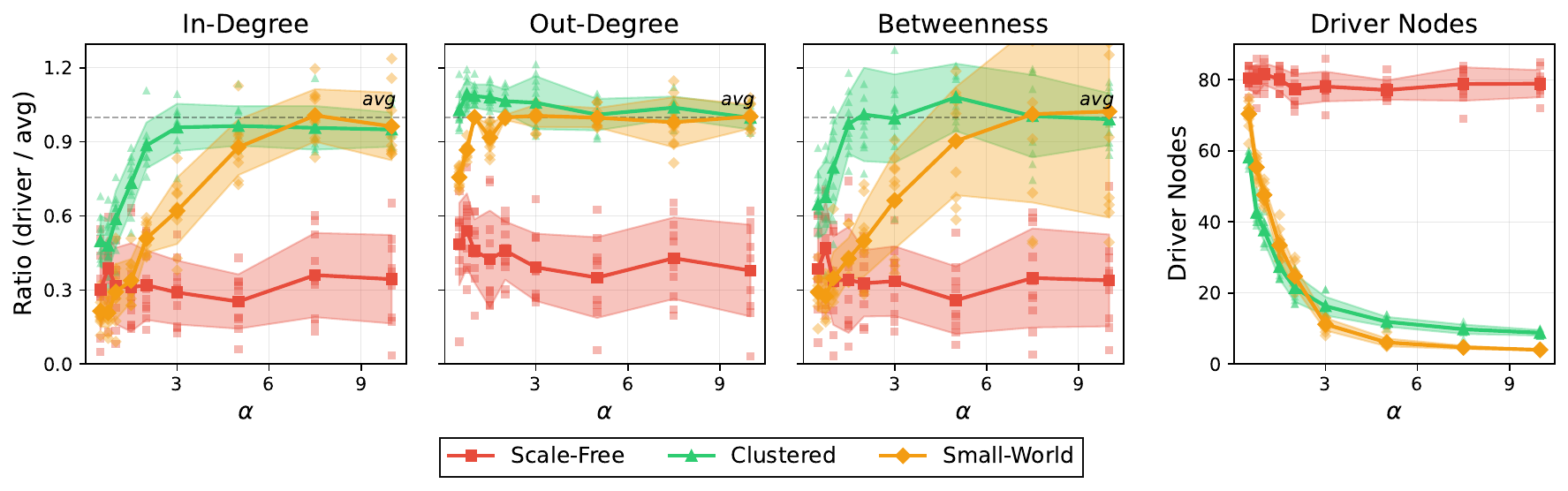}

    \caption{Structural properties of driver nodes in structured hypergraph topologies. Comparison of driver node characteristics across scale-free (red), clustered (green), and small-world (orange) hypergraph networks with $n=100$ varying densities. Solid lines show means over 10 independent realizations; shaded regions indicate standard deviations. Points represent individual trials.}
    \label{fig: exp 3}
\end{figure*}

\subsection{Large-Scale Hypergraphs}

Next, we evaluate driver node selection for large-scale hypergraphs. 
We consider state dimensions $n$ logarithmically spaced between 10 and 20,000, with $k=4$ and $\alpha \in \{0.5, 1.0\}$.  
Runtime and the number of driver nodes selected are reported in Fig. \ref{fig: EXP 2}.
For reference, the BFS runtime is reported for $n=20$ in the small-scale experiments. 
The results highlight the strong scalability of MaG. In particular, MaG exceeds the BFS runtime only when $n \gg 10{,}000$, corresponding to roughly a 500-fold improvement in scalability. Moreover, the runtime of MaG grows approximately linearly with $n$, indicating strong performance for large-scale hypergraphs. Greedy selection exhibits similar scaling behavior but remains substantially slower in practice; for example, when $\alpha = 0.5$ and $n = 20{,}000$, greedy selection is approximately two orders of magnitude slower than MaG. Matching-based selection achieves the fastest runtime overall, but it consistently fails to select a sufficient set of driver nodes to guarantee structural controllability. 
The memory requirements of MaG scale linearly with both $n$ and the number of hyperedges $m$, and are therefore omitted for brevity. Overall, these results demonstrate that MaG provides an efficient and scalable approach to driver node selection while guaranteeing structural controllability for large-scale hypergraph systems.

\subsection{Driver Node Placement in Structured Hypergraphs}

Many large-scale networked systems exhibit nontrivial structural organization, such as scale-free degree distributions, modular clustering, and small-world connectivity patterns. Because quantitative interaction strengths are often unavailable or unreliable in experimental settings, especially in biological and social systems, it is important to understand how network topology alone influences structural controllability and the resulting intervention requirements. In particular, topology-driven analysis can reveal where control inputs should be placed even when the exact dynamics or interaction weights are unknown. To investigate this relationship, we generate synthetic hypergraph networks with $n=100$ nodes across a range of densities $\alpha$, under three representative topologies. The first class consists of scale-free networks generated using preferential attachment, which capture hub dominated structures commonly observed in gene regulatory networks and protein protein interaction systems. The second class consists of clustered or modular networks with strong intra module connectivity and relatively sparse inter module connections, reflecting the modular organization often seen in biological pathways or functional subsystems. The third class consists of small-world networks that combine strong local clustering with a small number of long range connections, a structure frequently reported in neural and communication systems. For each topology and density level, we generate 10 independent realizations and apply MaG to determine the corresponding minimum set of driver nodes required to guarantee structural controllability.

The results shown in Fig. \ref{fig: exp 3} reveal clear topology dependent patterns in driver node placement. In scale-free networks, MaG tends to select nodes with lower in degree and betweenness centrality relative to the network average. In contrast, for clustered and small-world networks, the selected driver nodes typically exhibit near average or slightly above average centrality. These observations suggest that in hub dominated systems, controllability can often be achieved by actuating peripheral nodes rather than highly connected hubs. Such a strategy may be advantageous in practical interventions, since targeting peripheral nodes can reduce unintended effects on the broader network. We further examine how the required number of driver nodes varies with network density. In scale free networks, the number of drivers remains relatively stable as $\alpha$ increases, indicating that additional edges often reinforce the influence of existing hubs without substantially altering controllability requirements. In contrast, clustered and small world networks exhibit more pronounced reductions in the number of required drivers as density increases, suggesting that additional connectivity can significantly enhance accessibility across modules or neighborhoods. These trends are consistent across independent realizations and depend primarily on connectivity patterns rather than parameter values. This reinforces the usefulness of structural controllability analysis in biological and other data limited applications, where network topology can often be inferred even when interaction strengths are uncertain or unavailable.

\section{Discussion}\label{sec:5.5}

The structural controllability framework developed in this work establishes a topology-based characterization of controllability for homogeneous polynomial systems induced by hypergraphs and suggests several natural extensions to broader classes of polynomial systems and control objectives.

\subsection{Even-Degree Systems}
The analysis in this work focuses on homogeneous polynomial systems whose vector field has odd degree, which ensures that controllability can be characterized through Lie algebraic rank conditions under generic parameter choices.
In contrast, polynomial systems with even-degree vector fields are generally not strongly controllable due to symmetry properties that restrict reachable sets.
Nevertheless, a weaker property, strong accessibility, can still be characterized for even-degree systems using the structural framework developed here.
A nonlinear control system is strongly accessible from the origin if the reachable set from the origin has a nonempty interior \cite{hermann1977nonlinear, jurdjevic1985polynomial}.
While strong accessibility does not imply controllability, it guarantees that the system can reach an open set of states from the origin, and serves as a necessary condition for many feedback stabilization and motion planning objectives.
The structural criterion of Proposition 6 applies directly in this setting.
Specifically, if the directed hypergraph $\Hs(\At, \Bv)$ associated with an even-degree homogeneous polynomial system contains no inaccessible state nodes and no hyperedge dilations, then the system is structurally strongly accessible from the origin.
The MaG algorithm can therefore be applied without modification to identify a minimal set of driver nodes that guarantees strong accessibility for even-degree hypergraph dynamics.
This extends the practical utility of the proposed framework to a broader class of polynomial systems, including those arising from even-order  interactions.

\subsection{Non-Uniform Hypergraphs}
Many real-world networked systems exhibit interactions of heterogeneous cardinality, which are naturally modeled as non-uniform hypergraphs with non-homogeneous polynomial dynamics. The structural framework developed here applies to such systems from two complementary perspectives.
First, when the highest-degree term of the dynamics has odd degree, structural controllability of the full system can be inferred from the homogeneous subsystem induced by the subhypergraph containing only the highest-cardinality edges. It is a classical result that for polynomial control systems, controllability is generically determined by the highest-degree terms of the vector field, as lower-degree terms contribute to the Lie algebra without obstructing the rank condition already satisfied by the leading term~\cite{jurdjevic1985polynomial}. Consequently, if that homogeneous subsystem is structurally controllable, then the full non-homogeneous system is structurally controllable for almost all parameter choices. The MaG algorithm need only be applied to the subhypergraph of highest-cardinality edges to verify this condition.
Second, a non-uniform hypergraph can be embedded into the homogeneous framework directly by requiring each hyperedge to correspond to an odd-degree monomial term.
This is achieved by allowing nodes to appear with multiplicity within a hyperedge. Under this construction, the non-uniform hypergraph is represented by a single even-order tensor whose induced homogeneous polynomial has odd degree, and whose controllability is governed by Proposition~2. The appropriate perspective depends on the structure of the underlying physical system.

\subsection{Target Control}

In many applications, the objective of control is not to influence all system states but rather to regulate a subset of state variables \cite{moore1981principal}.
This setting corresponds to the problem of target controllability, where the goal is to steer only a designated set of state variables \cite{gao2014target}.
The hypergraph structural framework extends to this scenario.
Let $\Ts\subseteq\Vs$ denote the set of target nodes.
Structural target controllability requires that each node in $\Ts$ be reachable from the control inputs via a directed hypergraph walk.
Accessibility conditions can therefore be evaluated with respect to the target set rather than the full node set.
Dilation constraints must also be considered relative to the target nodes.
If a subset of target nodes forms a hyperedge dilation, then the available hyperedges provide insufficient independent influence to control those states.
Eliminating such dilations requires either additional driver nodes or modifications to the placement of control inputs. Accordingly, structural target controllability can be characterized by requiring that all target nodes are accessible and that no dilation exists within the hypergraph structure governing the target set.
From an algorithmic perspective, the driver node selection procedure proposed in this work can be adapted to this setting by restricting accessibility checks to the target nodes while maintaining dilation detection on the associated hypergraph.

\section{Conclusion}\label{sec:5}
Structural controllability has long served as a foundational tool in systems and control by enabling controllability analysis and actuator placement.
The original framework proposed by Lin provides a complete characterization for systems with pairwise interactions and catalyzed scalable methods for network control.
In this work, we extend the framework to higher-order systems with nonlinear and group interactions.
We develop a hypergraph-based criterion and algorithms to verify structural hypergraph controllability that allow for evaluating classical conditions of nonlinear control on large-scale networked systems. Specifically, we consider a class of nonlinear control systems with odd-degree homogeneous polynomial dynamics, which admit a tensor-based and hypergraph representations to describe nonlinear and group interactions.
By defining hypergraph structure through the support pattern of polynomial terms, we establish a directed hypergraph representation that encodes how multi-way interactions propagate control influence.
Our main theoretical result provides a necessary and sufficient structural characterization of strong controllability. A polynomial system is structurally controllable if and only if its associated directed hypergraph contains no inaccessible state nodes and no hyperedge dilations.
This result generalizes Lin’s accessibility and dilation conditions from graphs to higher-order interaction structures and offers a computationally tractable alternative to direct evaluation of the Lie algebraic rank conditions of polynomial control.
Moreover, we derive a matching-based lower bound on the minimum number of driver nodes required to achieve structural controllability and proposed scalable algorithms for driver-node selection.
By combining maximum matching on a star expansion with an accessibility completion procedure based on directed hypergraph walks, the resulting method yields driver sets that satisfy the structural controllability conditions while remaining computationally feasible for high dimensional systems.
Numerical experiments on synthetic hypergraphs demonstrated that the proposed approach scales efficiently to systems with tens of thousands of states and interactions and selects near-minimal driver sets.

The proposed framework admits extensions beyond controllability. An analogous hypergraph-based conditions can be developed for structural observability, where work leveraging directed hypergraphs as computational graphs has already begun \cite{montanari2022functional, pickard2025dynamic, pickard2025scalable, zhang2025data}. More broadly, the hypergraph representation introduced here provides a foundation for extending structural systems concepts to higher-order nonlinear networked systems. The proposed framework and associated algorithms enable the data-driven analysis and control of complex natural and engineered networks. By adopting a structural perspective, the framework mitigates limitations arising from parametric uncertainty and facilitates controllability analysis of systems with heterogeneous, nonlinear, and partially observed higher-order interactions. Such systems arise across biomedical, ecological, social, and communication domains, highlighting the potential for broad applicability of structural controllability analysis for large-scale higher-order networks.

\section*{Acknowledgments}
Joshua Pickard and Can Chen thank Prof. Indika Rajapakse, Prof. Anthony Bloch, and Dr. Amit Surana for introducing them to this area of study.

\section*{References}
\bibliographystyle{IEEEtran}
\bibliography{references}

\begin{IEEEbiography}[{\includegraphics[width=1in,height=1.25in,clip,keepaspectratio]{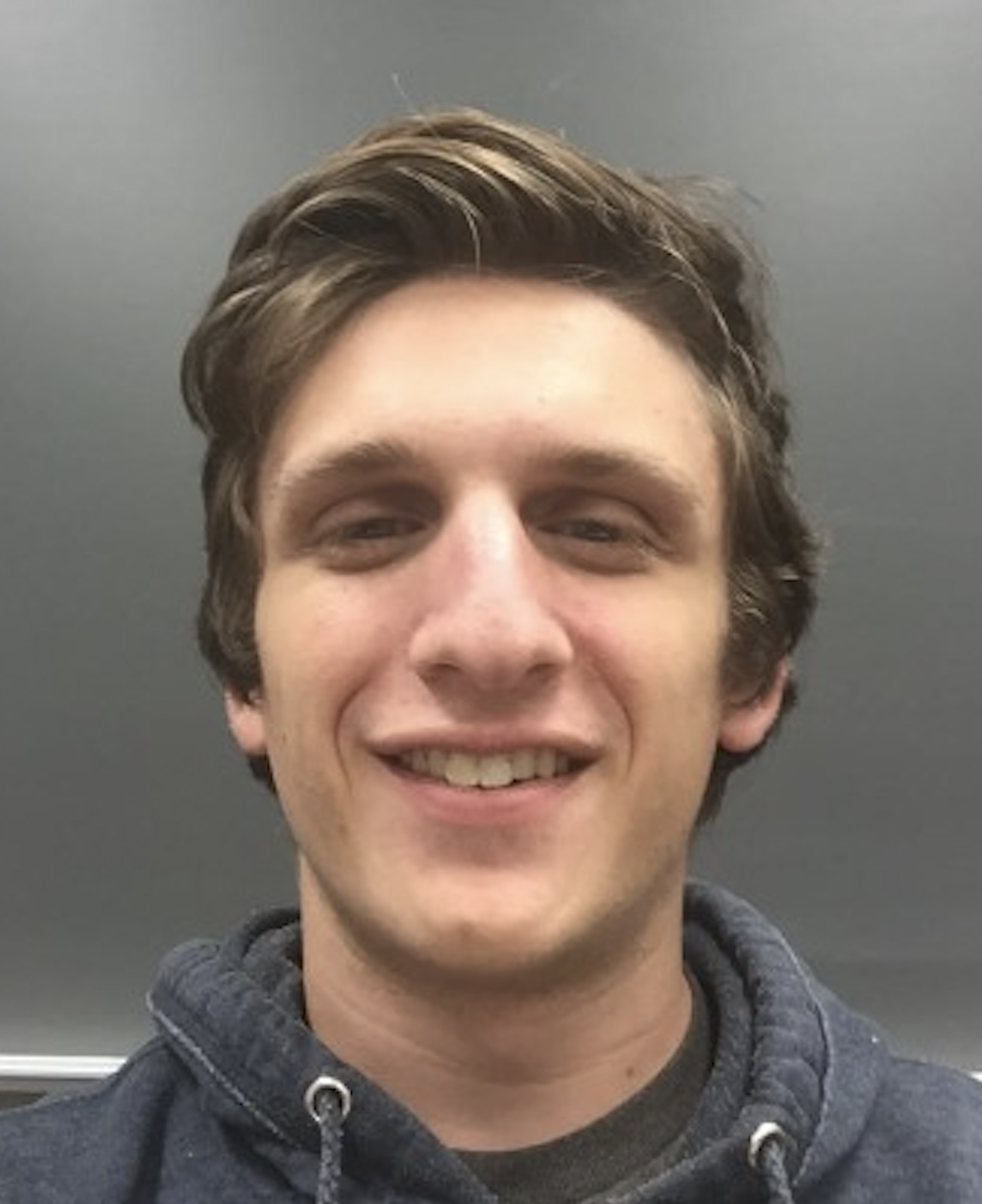}}]{Joshua Pickard} is a Eric and Wendy Schmidt Center postdoctoral fellow at the Broad Institute of MIT and Harvard, Cambridge, MA, USA. He received the B.S.E. degree in computer science in 2022 and the Ph.D. degree in bioinformatics in 2025, both from the University of Michigan, Ann Arbor, MI, USA, where his research focused on observability and higher-order methods for biomarker discovery. His research interests include data-driven approaches for analyzing the dynamics and control of large-scale biomedical systems.
\end{IEEEbiography}

\begin{IEEEbiography}[{\includegraphics[width=1in,height=1.25in,clip,keepaspectratio]{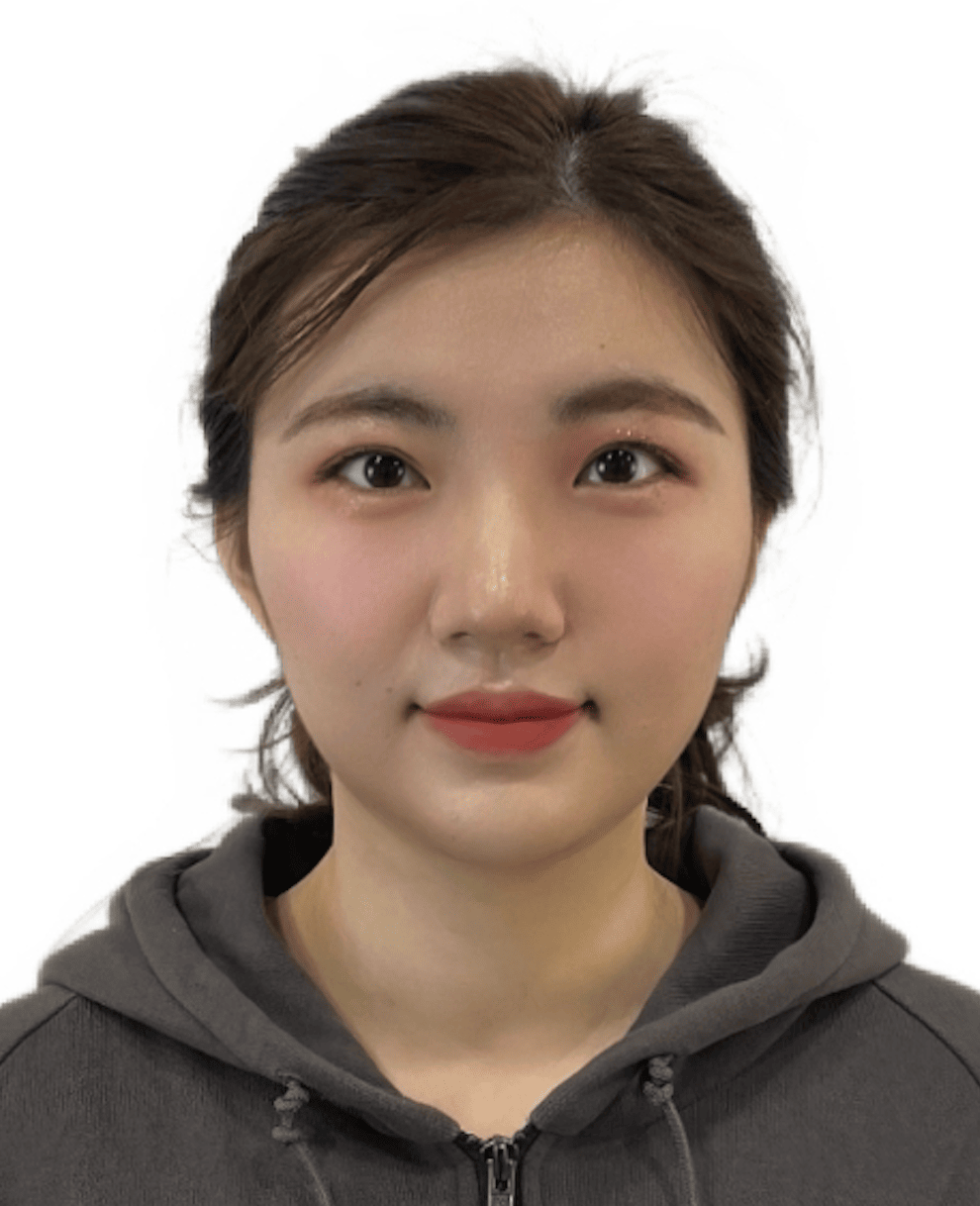}}]{Xin Mao}
received the B.Eng. degree in Automation Engineering from Zhejiang University, Hangzhou, Zhejiang, China, in 2017, and Ph.D. degree in
Electronic and Computer Engineering from Hong Kong University of Science and Technology, Hong Kong SAR, China, in 2022. She is currently a Postdoctoral Research Associate in the School of Data Science and Society at the University of North Carolina at Chapel Hill, Chapel Hill, NC, USA.
\end{IEEEbiography}

\begin{IEEEbiography}[{\includegraphics[width=1in,height=1.25in,clip,keepaspectratio]{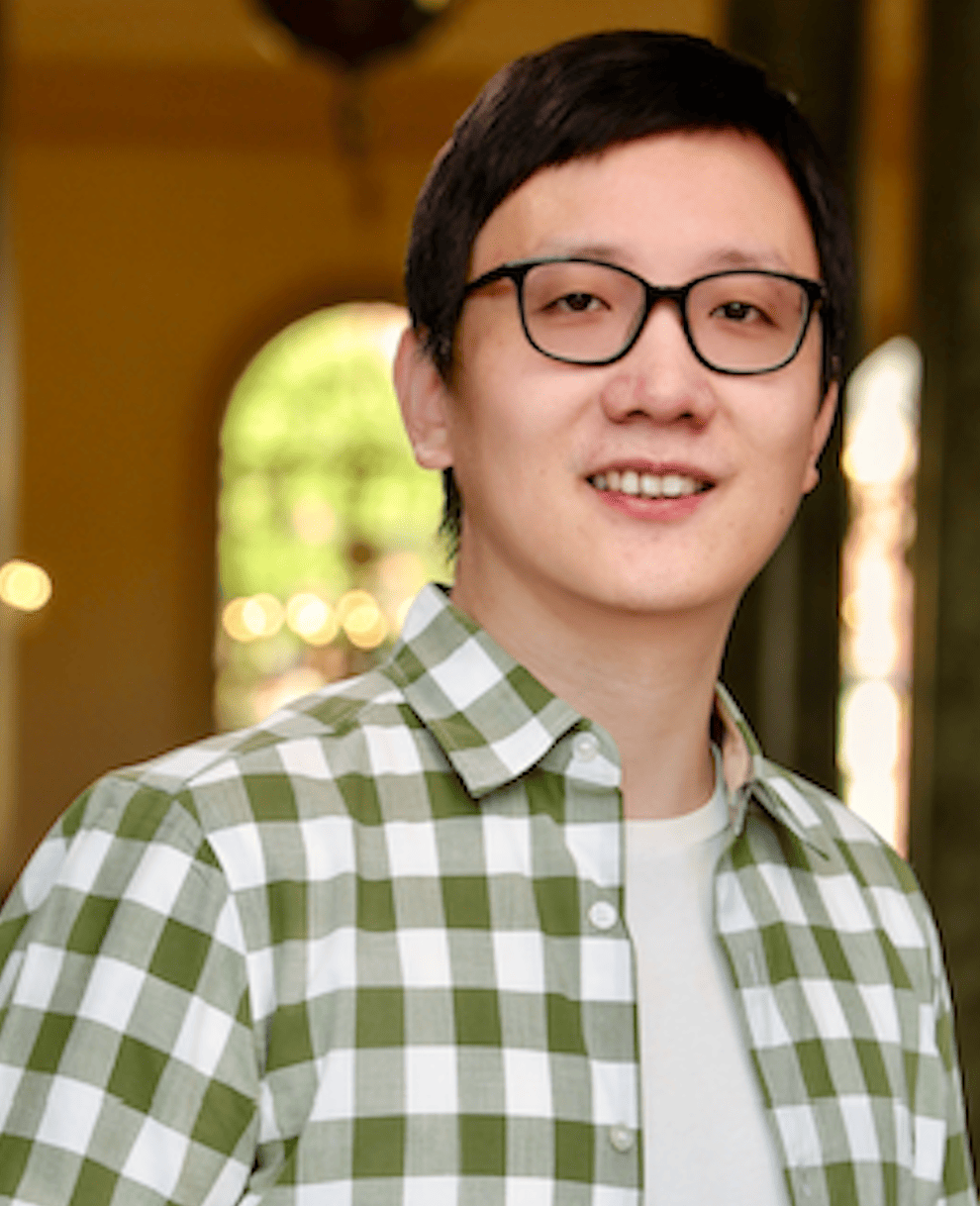}}]{Can Chen}
is currently an Assistant Professor in the School of Data Science and Society, with secondary appointments in the Department of Mathematics and the Department of Biostatistics, at the University of North Carolina
at Chapel Hill, Chapel Hill, NC, USA. He received the B.S. degree
in Mathematics from the University of California, Irvine, Irvine, CA, USA, in 2016, the M.S. degree in Electrical and Computer Engineering in 2020, and the Ph.D. degree in Applied and Interdisciplinary Mathematics in 2021, both from the University of Michigan, Ann Arbor, MI, USA. From 2021 to 2023, he was a Postdoctoral Research Fellow in the Channing Division of Network Medicine at Brigham and Women’s Hospital and Harvard Medical School, Boston, MA, USA. His research interests include control theory, network science, machine learning, and bioinformatics. 
\end{IEEEbiography}

\vfill

\end{document}